\documentclass[12pt]{amsart} 
\usepackage{amsfonts,graphics,amsmath,amsthm,amsfonts,amscd, amssymb,amsmath,latexsym}
\usepackage{epsfig}
\usepackage{flafter}

\theoremstyle{plain}
\newtheorem{theorem}{Theorem}[section]
\newtheorem{corollary}[theorem]{Corollary}
\newtheorem{lemma}[theorem]{Lemma}
\newtheorem{claim}[theorem]{Claim}

\newtheorem{proposition}[theorem]{Proposition}
\newtheorem{definition-lemma}[theorem]{Definition-Lemma}

\newtheorem{remark}[theorem]{Remark}

\def\ideal#1.{I_{#1}}
\def\ring#1.{\mathcal {O}_{#1}}
\def\fring#1.{\hat{\mathcal {O}}_{#1}}
\def\proj#1.{\mathbb P(#1)}
\def\pr #1.{\mathbb P^{#1}}
\def\af #1.{\mathbb A^{#1}}
\def\Hz #1.{\mathbb F_{#1}}
\def\Hbz #1.{\overline{\mathbb F}_{#1}}
\def\pic#1.{\operatorname {Pic}\,(#1)}
\def\pico#1.{\operatorname{Pic}^0(#1)}
\def\picg#1.{\operatorname {Pic}^G(#1)}
\def\ner#1.{NS (#1)}
\def\rdown#1.{\llcorner#1\lrcorner}
\def\rup#1.{\ulcorner#1\urcorner}
\def\cone#1.{\operatorname {NE}(#1)}

\def\ccone#1.{\overline{\operatorname {NE}}(#1)}
\def\coef#1.{\frac{(#1-1)}{#1}}
\def\vit#1.{D_{\langle #1 \rangle}}
\def\mm#1.{\overline {M}_{0,#1}}
\def\H1#1.{H^1(#1,{\ring #1.})}
\def\ac#1.{\overline {\mathbb F}_{#1}}

\def\adj#1.{\frac {#1-1}{#1}}
\def\spn#1.{\overline{#1}}
\def\ses#1.#2.#3.{0\to #1\to #2\to #3 \to 0}
\def\pek#1.#2.{\Cal P^{#1}(#2)}
\def\plk#1.#2.{\Cal P^{\leq #1}(#2)}
\def\ev#1.{\operatorname{ev_{#1}}}
\def\bminv#1.{(\nu_1,s_1;\nu_2,s_2;\dots ;\nu_{#1},s_{#1};\nu_{r+1})}
\def\zinv#1.{(\nu_1,s_1;\nu_2,s_2;\dots ;\nu_{#1},s_{#1};0)}
\def\iinv#1.{(\nu_1,s_1;\nu_2,s_2;\dots ;\nu_{#1},s_{#1};\infty)}
\def\map#1.#2.{#1 \longrightarrow #2}
\def\rmap#1.#2.{#1 \dasharrow #2}
\def\emb#1.#2.{#1 \hookrightarrow #2}

\def\N{\mathbb N}

\def\e{\Cal E}

\def\e1{E_1}
\def\e2{E_2}

\def\bB{\mathbf B}

\newcommand\Q{{\mathbb{Q}}}
\newcommand\R{{\mathbb{R}}}
\begin{document}
\title{Invariance of certain plurigenera for surfaces in mixed
 characteristics}
\author{Andrew Egbert}

\author{Christopher D. Hacon} 
\address{Department of Mathematics \\  
University of Utah\\  
Salt Lake City, UT 84112, USA}
 \thanks{The second author
 was supported by NSF research grants no: DMS-1300750, DMS-1265285 and by
  a grant from the Simons Foundation; Award Number: 256202.}

\maketitle
\begin{abstract} We prove the deformation invariance of Kodaira dimension and of certain plurigenera for log surfaces which are smooth over a DVR. \end{abstract}
\section{Introduction}
If $f:X\to T$ is a smooth projective morphism of complex quasi-projective varieties, then by a celebrated theorem of Siu (\cite{Siu98}, \cite{Siu02}), it is known that the plurigenera of the fibers $P_m(X_t):=h^0(mK_{X_t})$ are independent of the point $t\in T$. This result (and its generalizations to log pairs)  is a fundamental fact of great importance in higher dimensional birational geometry. It plays a fundamental role in the construction of moduli spaces of varieties of log general type. Unluckily, this result does not generalize even to families of surfaces over a curve (or a discrete valuation ring DVR).
In \cite{La83}, it is shown that $P_1$ is not deformation invariant for Enriques surfaces in characteristic $2$. In \cite{KU85}, it is shown that in fact the deformation invariance of plurigenera does not hold for certain elliptic surfaces and in \cite{Suh08}, there are examples of smooth families of surfaces of general type over any DVR for which $P_1$ is non constant (and in fact its value can jump by an arbitrarily big amount). On the positive side, in \cite{KU85} it is shown that 
if $X\to {\rm Spec }(R)$ is a smooth family of surfaces over a DVR in positive or mixed characteristic, then one can run the minimal model program for $X$ (over $R$). As a consequence of this, it is observed that $\kappa (X_K)=\kappa (X_k)$ where $k$ is the residue field and $K$ is the fraction field of $R$. It should be noted that the minimal model program is established for semistable families of surfaces in positive or mixed characteristic  (see  \cite{Kaw94}), for log canonical surfaces over excellent base schemes (see \cite{Tan16}) and for $3$ folds over a field $k$ of characteristic $p\geq 7$ (see \cite{HX13} and \cite{Bir15}).
In this paper (Theorem \ref{t-definv} and Corollary \ref{c-mmp}), we generalize the result of Katsura and Ueno to log surfaces (smooth over a DVR) and we show the deformation invariance of certain plurigenera.
\begin{theorem}Let $(X,B)$ be a klt pair which is log smooth, projective of dimension $2$ over a DVR $R$ with perfect residue field $k$ of characteristic $p > 0$, then $\kappa (K_{X_k}+B_k)=\kappa (K_{X_K}+B_K)$  and we may run the minimal model program with scaling of an ample divisor $H$. 
Moreover, if either $\kappa (K_{X_k}+B_k)\ne 1$ or $\kappa (K_{X_k}+B_k)= 1$ and $B_k$ is big over ${\rm Proj}R(K_{X_k}+B_k)$, then there exists an integer $m_0>0$  such that for any positive integer $m\in m_0\N$ we have
$$h^0(m (K_{X_K}+B_K))=h^0(m (K_{X_k}+B_k)).$$\end{theorem}
The strategy is to reduce the proof of the above theorem to the case when $(X_k,B_k)$ is terminal and $\mathbf B (K_{X_k}+B_k)$ contains no components of the support of $B_k$. In this case we observe that the steps of a $K_{X_k}+B_k$ mmp are also steps of a $K_{X_k}$ mmp and we are thus able to deduce the result from \cite{KU85}.
\begin{remark} Many results and techniques in this paper were developed in the first author's Ph.D.'s thesis \cite{Egb16}.\end{remark}
%%%%%%%%%%%%%%%%%%%%%%%%%%%%%%%%%%%%
\section{Preliminaries}
Let $X$ be a normal projective variety, ${\rm WDiv} (X)$ the group of Weil divisors. If $B=\sum b_iB_i\in {\rm WDiv} _\Q(X)$ is a $\Q$-divisor on $X$, then $\lfloor B \rfloor=\sum \lfloor b_i\rfloor B_i$
where $\lfloor b_i\rfloor={\rm max}\{n\in \mathbb Z|n\leq b_i\}$. We denote $\{B\}=B-\lfloor B \rfloor$ and $|B|=|\lfloor B \rfloor|+\{B\}$ where $$|\lfloor B \rfloor|=\{D\in {\rm WDiv} (X)|
 D\geq 0,\ D-\lfloor B \rfloor=(f),\ f\in K(X)\}.$$ The {\bf stable base locus} of $B$ is $\bB (D)=\cap _{m\in \N} {\rm Bs}(mD)$.
If $D_1,\ldots , D_r$ are $\Q$-divisors on a normal projective variety $X$, then $$R(D_1,\ldots , D_r):=\bigoplus _{m_1,\ldots , m _r\in \N}H^0(\mathcal O _X(\sum _{i=1}^r m_iD_i)).$$ 
Let $(X,B)$ be a {\bf pair} so that $X$ is normal, $0\leq B$ is a $\Q$-divisor and $K_X+B$ is $\Q$-Cartier. If $\nu :X'\to X$ is a proper birational morphism, then we write $K_{X'}+B_{X'}=\nu ^*(K_X+B)$. We say that $(X,B)$ is {\bf Kawamata log terminal or klt} (res. {\bf terminal}) if for any log resolution $\lfloor B_{X'}\rfloor \leq 0$ (resp. $\lfloor B_{X'}\rfloor \leq E$ where $E$ denotes the reduced exceptional divisor). 
We let $\mathbf M _B$ be the $b$-divisor defined by the sum of the strict transform of $B$ and the exceptional divisors (over $X$). 
We refer the reader to \cite{KM98} and \cite{BCHM10} for the standard definitions of the minimal model program including extremal rays, flipping and divisorial contractions, running a minimal model program with scaling, log terminal and weak log canonical models.

\begin{theorem}\label{t-ct} Let $(X,B)$ a 2-dimensional projective klt pair over an algebraically closed field $k$. Then 
$$\overline {NE}(X)=\overline {NE}(X)_{K_X+B\geq 0}+\sum _{i\in I}\R _{\geq 0}C_i$$
where $I$ is countable, $(K_X+B)\cdot C_i<0$, $C_i$ is rational and $C_i^2<0$.
If $H$ is an ample $\Q$-divisor on $X$, then the set $\{ i\in I|(K_X+B+H)\cdot C_i\leq 0\}$ is finite.
\end{theorem}
\begin{proof} See \cite[3.13, 3.15]{Tan14}.
\end{proof}
\begin{lemma}\label{l-dc} Let  $X$ be a surface over an algebraically closed field $k$ and $(X,B)$ a projective klt pair.
If $R$ is a $K_X+B$ negative extremal ray, then there exists a proper morphism $f:X\to X'$ such that $f_*\mathcal O _X=\mathcal O _{X'}$, $f$ contracts a curve $C\subset X$ if and only if $[C]=R$.\end{lemma}
\begin{proof} See \cite[3.21]{Tan14}.
\end{proof}
\begin{theorem}\label{t-mmps} Let $X$ be a projective surface over a perfect field $k$. Assume that $(X,B)$ is klt and $H$ is an ample $\mathbb Q$-divisor on $X$. Then \begin{enumerate}
\item the ring $R(K_X+B,K_X+B+H)$ is finitely generated, and
\item if $K_X+B$ is pseudo-effective, then there exists a birational morphism $\nu :X\to X'$ and a constant $\epsilon >0$ such that $\nu$ is a $K_X+B+tH$ minimal model for any $0\leq t\leq \epsilon$.\end{enumerate}\end{theorem}
\begin{proof} By \cite{Tan14}, the $K_X+B$ mmp with scaling of $H$ terminates. Therefore there is a sequence of rational numbers $t_0\leq t_1\leq \ldots \leq t_n$ and morphisms $f_i:X\to X_i$ (induced by successive divisorial contractions) such that \begin{enumerate} \item if $t_0>0$, then $\kappa (K_X+B+tH)<0$ for any $0\leq t< t_0$,
\item $R(K_X+B+t_{0}H,
K_X+B+t_iH)\cong R( K_{X_i}+B_i+t_{0}H_i,
K_{X_i}+B_i+t_iH_i)$ where $B_i=f_{i,*}B$ and $H_i=f_{i,*}H$,
\item  $K_{X_i}+B_i+t_{i-1}H_i$ and
$K_{X_i}+B_i+t_iH_i$ are both nef (and hence semiample).\end{enumerate} 
It then follows easily that $$R:= \oplus _{i=1}^n R(K_{X_i}+B_i+t_{i-1}H_i,K_{X_i}+B_i+t_iH_i)$$ is finitely generated (cf. \cite[2.7]{HX15}).
Since $R$ surjects on to $R(K_X+B,K_X+B+H)$, this ring is also finitely generated. 

Finally, if $K_X+B$ is pseudo-effective, then $t_0=0$, $K_{X_1}+B_1+tH_1$ is nef for $0\leq t\leq t_1$ and $X\to X_1$ is given by a sequence of $K_X+B+tH$ divisorial contractions for any $0\leq t <t_1$.\end{proof}  

\begin{proposition}\label{p-mmps} Let $X$ be a projective surface over an algebraically closed field $k$. Assume that $(X,B)$ is a klt pair and $\nu :X'\to X$ is a proper birational morphism such that $(X',B')$ is terminal where $K_{X'}+B'=\nu ^*(K_X+B)$. Let $\Theta =B'-B'\wedge N_\sigma (K_{X'}+B')$ and $\phi ':X'\to X'_M$ the minimal model for $(X',\Theta )$. If $\phi:X\to X_M$ is the minimal model for $(X,B)$, then the rational map $\mu:X'_M\to X_M$ is a morphism and $K_{X'_M}+\phi '_*\Theta =\mu ^*(K_{X_M}+\phi _* B)$.
If $\kappa (K_X+B)=1$ and $B$ is big over ${\rm Proj}R(K_X+B)$, then $\Theta $ is big over ${\rm Proj}R(K_{X'}+\Theta)$.
\end{proposition}
\begin{proof} Consider the morphism $\psi:X'\to X_M$. Since $K_X+B=\phi ^*(K_{X_M}+\phi _*B)+E$, then $K_{X'}+B'=\psi ^*(K_{X_M}+\phi _*B) +\nu ^*E$ where $K_{X_M}+\phi _*B$ is nef and $\nu ^*E$ is effective and $\psi$ exceptional. It follows that 
$N_\sigma (K_{X'}+B')=\nu ^*E$ and so $K_{X'}+\Theta =\psi ^* (K_{X_M}+\phi _*B)+E'$ where $0\leq E'\leq \nu ^*E$. In particular the divisors contracted by $\phi '$ are precisely the divisors contained in ${\rm Supp} (E')$ and so $X'\to X_M$ factors through $\phi '$. 
We have $K_{X'_M}+\phi '_* \Theta =\mu ^*(K_{X_M}+\phi _*B)+\phi '_* E'$ where $\mu _* (\phi '_*E')\leq \phi _* E=0$ and hence $\phi'_*E $ is $\mu$ exceptional. By the negativity lemma, it follows that $\phi '_* E'=0$.

Note that since $H^0(m(K_X+B))\cong H^0(m(K_{X'}+\Theta ))$ for all $m\geq 0$, it follows that $Z:={\rm Proj}R(K_X+B)={\rm Proj}R(K_{X'}+\Theta)$. The bigness of $B$ over $Z$ is equivalent to $B\cdot X_z>0$ for general $z\in Z$. But then
$$\Theta \cdot X'_z=\mu _* \phi '_*\Theta \cdot (X'_M)_z=\phi _*B\cdot (X'_M)_z=B\cdot X_z>0$$ and so $\Theta$ is big over $Z$.
\end{proof}

In what follows $R$ will denote a DVR with residue field $k$ and fraction field $K$. If $f:X\to {\rm Spec}(R)$ is a morphism, then we let $X_K=X\times _{{\rm Spec}(R)}{\rm Spec}(K)$ be the generic fiber and $X_k= X\times _{{\rm Spec}(R)}{\rm Spec}(k)$ be the special fiber. We say that a pair $(X,B)$ is log smooth over $R$ if $X$ and each strata of the support of $B$ are smooth over ${\rm Spec}( R)$.
%\begin{lemma}\label{l-cbc} Let $f:X\to {\rm Spec }(R)$ be a projective morphism from a normal variety to a DVR. If $\FF$ is a coherent sheaf on $X$ which is flat over ${\rm Spec}(R)$, then $\chi _k(\FF _k)=\chi _K(\FF _K)$. Moreover, for any $i\geq 0$ the following are equivalent. \begin{enumerate}\item $\dim _kH^i(\FF _k)=\dim _KH^i(\FF _K)$, \item $H^i(\FF )$ is free over ${\rm Spec}(R)$ and $H^i(\FF )\otimes _R k\to H^i(\FF _k)$ is an isomorphism, \item $H^{i+1}(\FF)$ is a torsion free $R$-module. \end{enumerate} \end{lemma} \begin{proof} \cite[5.3.20, 5.3.22]{Liu02}. [xxxcheck] \end{proof}
\begin{lemma}\label{l-open} Let $f:X\to {\rm Spec}(R)$ be a smooth projective morphism from a normal variety to a DVR and $L$ a line bundle on $X$, then $L$ is ample (resp. nef) if and only if so is $L_k:=L|_{X_k}$.\end{lemma}
\begin{proof} Clearly, if $L$ is ample or nef, then so is $L_k$. It is well known that ampleness is an open condition and so if $L_k$ is ample then so is $L$. Finally if $L_k$ is nef and $H$ is ample, then $L_k+tH_k$ is ample for any $t>0$, so that $L+tH$ is ample and hence $L$ is nef.
\end{proof}
\begin{lemma} Let $(X,B)$ be a log pair which is log smooth over ${\rm Spec }(R)$ where $R$ is a DVR. If $R\subset \tilde R$ is an inclusion of DVR's, then $(X_{\tilde R},B_{\tilde R})$ is log smooth over $ {\rm Spec }(\tilde  R)$. If $(X,B)$ is terminal (resp. klt), then so is $(X_{\tilde R},B_{\tilde R})$.\end{lemma}
\begin{proof}Since smoothness is preserved by base change, it follows that $(X_{\tilde R},B_{\tilde R})$ is log smooth over $ {\rm Spec }(\tilde R)$. The pair $(X,B)$ is klt (resp. terminal) if and only if the coefficients of $B$ are $<1$ (resp.  the coefficients of $B$ are $<1$ and if two components intersect, then the sum of the coefficients is $<1$). Since there is a one to one correspondence of components of $B$ with components of $B_k$, the lemma follows. \end{proof}
\begin{theorem}\label{t-KU} [Katsura-Ueno \cite{KU85}] Let $f:X\to {\rm Spec}(R)$ be an algebraic space which is proper, separated of finite type and two dimensional over $ {\rm Spec}(R)$ where $R$ is a DVR with algebraically closed residue field $k$ and field of fractions $K$. If $X_k$ contains an exceptional curve of the first kind $e\subset X_k$ then there exists a DVR $\tilde R\supset R$ with residue field $k$ and a surjective proper morphism $\pi :X\times {\rm Spec}(\tilde R)\to \tilde Y$ over ${\rm Spec}(\tilde R)$ where $\tilde Y\to {\rm Spec}(\tilde R)$  is proper, separated of finite type and two dimensional, $\pi _k$ contracts the exeptional curve of the first kind $e\subset X_k$ and $\pi _K:X_{\tilde K}\to \tilde Y_{\tilde K}$ is also a contraction of an exeptional curve of the first kind.\end{theorem}

%%%%%%%%%%%%%%%%%%%%%%%%%%%%%%%%%%%%%%%%%
\section{Main result}
\begin{theorem}\label{t-definv} Let $(X,B)$ be a klt pair which is log smooth, projective of dimension $2$ over a DVR $R$ with perfect residue field $k$ of characteristic $p>0$ and perfect fraction field $K$. If $K_X+B$ is $\Q$-Cartier then  $\kappa (K_{X_k}+B_k)=\kappa (K_{X_K}+B_K)$ and if either $\kappa (K_{X_k}+B_k)\ne 1$ or $\kappa (K_{X_k}+B_k)= 1$ and $B_k$ is big over ${\rm Proj}R(K_{X_k}+B_k)$, then there exists an integer $m_0$ such that for any positive integer  $m\in m_0\N$ we have
$$h^0(m (K_{X_K}+B_K))=h^0(m (K_{X_k}+B_k)).$$
\end{theorem}
\begin{proof} Consider an inclusion of DVR's $R\subset \tilde R$. If $\tilde k$ and $\tilde K$ denote the residue field and the fraction field of $\tilde R$, then $h^0(m (K_{X_k}+B_k))=h^0(m (K_{ X_{\tilde k}}+B_{\tilde k}))$ and $h^0(m (K_{X_K}+B_K))=h^0(m (K_{ X_{\tilde K}}+B_{\tilde K}))$. Thus we are free to replace $X\to R$ by a base  extension $\tilde X=X\times _R\tilde R\to \tilde R$. In particular we may assume that $k$ is algebraically closed.

If $h^0(m (K_{X_k}+B_k))=0$, then by semicontinuity $h^0(m (K_{X_K}+B_K))=0$. Therefore, the theorem holds trivially in the case $\kappa (K_{X_k}+B_k)=-\infty $.  Thus we may assume that $\kappa (K_{X_k}+B_k)\geq 0$.

\begin{claim}\label{c-1} The theorem holds under the additional assumption that $(X_k,B_k)$ is terminal and no component of the support of $B_k$ is contained in $\bB (K_{X_k}+B_k)$.\end{claim}
\begin{proof} Since $k$ is algebraically closed, then by the Cone Theorem (Theorem \ref{t-ct}) $$\overline {NE}(X_k)=\overline {NE}(X_k)_{K_{X_k}+B_k\geq 0}+\sum _{i\in I}\R _{\geq 0}C_i$$
where $I$ is countable, $(K_{X_k}+B_k)\cdot C_i<0$ and $C_i$ is rational and not  contained in the support of $B_k$. 
Notice in fact that if $C_i$ is contained in the support of $B_k$, then since  $(K_{X_k}+B_k)\cdot C_i<0$ we have $C_i\subset \bB (K_{X_k}+B_k)$, which we have assumed is impossible. Note then that $C_i\cdot B_k\geq 0$ and so $K_{X_k}\cdot C_i<0$ 
and hence if $C_i$ spans a $K_{X_k}+B_k$-negative extremal ray, then it also spans a  $K_{X_k}$-negative extremal ray and so it  can be contracted by a divisorial contraction of an exceptional curve of the first kind $X_k\to X'_k$. In particular $X'_k$ is 
also a smooth surface. %(Lemma \ref{l-dc}).  But then, since $X_k$ is terminal, so is $X'_k$ and in particular $X'_k$ is smooth. 
Thus we may assume that $C_i$ is an extremal curve of the first kind. By Theorem \ref{t-KU} (after extending $R$), we may assume that there is a morphism $X\to X'$ 
of smooth surfaces over $R$ such that $X_K\to X'_K$ also contracts  an extremal curve of the first kind.

Suppose now that $\nu :X\to \bar X$ is a morphism of smooth surfaces over ${\rm Spec }(R)$ such that $X_K\to \bar X_K$ and $X_k\to \bar X_k$ are given by a finite sequence of contractions of  extremal curves of the first kind such that the exceptional locus of $X_k\to \bar X_k$ contains no components of $B_k$.
Then $(\bar X_k,\bar B_k)$ is terminal and $K_{X_k}+B_k=\nu _k^*(K_{\bar X_k,}+\bar B_k)+F_k$ where $B_k=\nu _{k,*}^{-1}\bar B_k$ and $B_k\wedge F_k=0$. In particular $\bB (K_{ X_k}+ B_k)=\bB (K_{\bar X_k}+\bar B_k)+F_k$.
 Suppose that $C\subset \bar X_k$
is contained in $\bB (K_{\bar X_k}+\bar B_k)\cap {\rm Supp}(\bar B_k)$, then $\nu ^{-1}_*C\subset \bB (K_{X_k}+B_k)\cap {\rm Supp}(B_k)$ which is impossible. Therefore, if $K_{\bar X_k}+\bar B_k$ is not nef, we can continue to contract exceptional curves of the first kind. Since each contraction reduces the Picard number of the central fiber $X_k$ by one, this procedure must terminate after finitely many steps. We may therefore assume that $K_{\bar X_k}+\bar B_k$ is semiample. In particular $K_{\bar X_k}+\bar B_k$ is nef and hence so is $K_{\bar X}+\bar B$ (see Lemma \ref{l-open}).

Suppose now that $\nu (K_{X_k}+B_k)=2$. In this case $K_{\bar X_k}+\bar B_k$ is nef and big so that there exists an integer $m_0>0$ such that $h^i(m(K_{\bar X_k}+\bar B_k))=0$ for all $m\in m_0\N$ and all $i>0$ (see \cite[2.6]{Tan15}). By semicontinuity, we also have $h^i(m(K_{\bar X_K}+\bar B_K))=0$ for all $m\in m_0\N$. But then, by flatness,
$$h^0(m(K_{X_K}+ B_K))=h^0(m(K_{\bar X_K}+\bar B_K))=\chi (m(K_{\bar X_K}+\bar B_K))=\qquad$$ $$\qquad \chi (m(K_{\bar X_k}+\bar B_k))=h^0(m(K_{\bar X_k}+\bar B_k))=h^0(m(K_{X_k}+ B_k)).$$

Suppose that $\kappa (K_{X_k}+B_k)=0$. Then we have $K_{\bar X_k}+\bar B_k\sim _\Q 0$. By Lemma \ref{l-open}, it follows that $\pm (K_{\bar X_K}+\bar B_K)$ is nef  and hence that $K_{\bar X_K}+\bar B_K\equiv 0$. By \cite[1.2]{Tan14}, $K_{\bar X_K}+\bar B_K\sim _\Q 0$. Thus there exists an integer $m_0>0$ such that $m_0(K_{\bar X_K}+\bar B_K)\sim 0$ and $m_0(K_{\bar X_k}+\bar B_k)\sim 0$.
Thus $h^0(m(K_{X_K}+ B_K))=h^0(m(K_{X_k}+ B_k))$  for all $m\geq 0$ divisible by $m_0$.

Suppose that $\kappa (K_{X_k}+B_k)=1$. Since $ K_{\bar X_k}+\bar B_k$ is nef, so is $K_{\bar X_K}+\bar B_K$. In particular $\kappa (K_{\bar X_K}+\bar B_K)\geq 0$ and thus, by semicontinuity we have  $\kappa (K_{\bar X_K}+\bar B_K)\in \{0,1\}$. 
 Let $H$ be a sufficiently ample divisor on $\bar X$. Then 
$(K_{\bar X_K}+\bar B_K)\cdot H_K=(K_{\bar X_k}+\bar B_k)\cdot H_k>0$ so $K_{\bar X_K}+\bar B_K\not\equiv 0$.
Therefore  $\kappa (K_{\bar X_K}+\bar B_K)= 1$.

 Finally, suppose that  $\kappa (K_{X_k}+B_k)=1$ and $B_k$ is big over ${\rm Proj}R(K_{X_k}+B_k)$. 
Note that  $\bar B_k$ is also big over ${\rm Proj}R(K_{X_k}+B_k)$ and hence $\bar B_k+K_{\bar X_k}+\bar B_k$ is big. Thus
we may write $\bar B_k+K_{\bar X_k}+\bar B_k\sim _\Q \bar A_k+\bar E_k$ where $\bar A_k$ is ample and $\bar E_k$ is effective. For any rational number $0<\epsilon \ll 1$, the pair $(\bar X _k, \Delta _k=(1-\epsilon)\bar B_k+\epsilon \bar E_k)$ is Kawamata log terminal and so the corresponding multiplier ideal sheaf is trivial
$\mathcal J (\Delta_k)=\mathcal O _{\bar X _k}$. If $L=N=m(K_{\bar X_k}+\bar B_k)$, then $N$ is nef and not numerically equivalent to zero while $L-( K_{\bar X_k}+\Delta_k)\sim _\Q (m-1-\epsilon )(K_{\bar X_k}+\bar B_k)+  \epsilon \bar A_k$ is ample and so by \cite[0.3]{Tan15},
$H^i(\mathcal O _{\bar X_k}(m(l+1)(K_{\bar X_k}+\bar B_k)))=0$ for $i>0$ and $l\gg 0$.  By semicontinuity, $H^i(\mathcal O _{\bar X_K}(m(l+1)(K_{\bar X_K}+\bar B_K)))=0$ for $i>0$ and $l\gg 0$ and hence $h^0(\mathcal O _{\bar X_k}(m(l+1)(K_{\bar X_k}+\bar B_k)))=h^0(\mathcal O _{\bar X_K}(m(l+1)(K_{\bar X_K}+\bar B_K)))$.
\end{proof}
We will now consider the general case. 
Since $(X,B)$ is log smooth over $R$, there is a sequence of blow ups along strata of $\mathbf M_B$ say $\nu : X'\to X$ such that $K_{X'}+B'=\nu ^*(K_X+B)$ is terminal and in particular $B'\geq 0$ and $(X',B')$ is log smooth. Since $R(K_{X'_k}+B'_k)\cong R(K_{X_k}+B_k)$ is finitely generated, $N_\sigma (K_{X'_k}+B'_k)$ is a $\Q$-divisor and hence so is $$\Theta _k:=B'_k-(B'_k\wedge N_\sigma (K_{X'_k}+B'_k)).$$
Note that $R(K_{X'_k}+\Theta _k)\cong R(K_{X'_k}+B'_k)$, $(X'_k, \Theta _k)$ is terminal and no component of $\Theta _k$ is contained in $\bB (K_{X'_k}+\Theta _k)$ \cite[2.8.3]{HMX14} and \cite[2.4]{HX13}.
Let $\Theta $ be the unique $\Q$-divisor supported on $B'$ such that $\Theta |_{X'_k}=\Theta _k$.
We remark that if $\kappa (K_{X_k}+B_k)=1$ and $B_k$ is big over ${\rm Proj}R(K_{X_k}+B_k)$, then by Proposition \ref{p-mmps}, $\Theta _k$ is big over  ${\rm Proj}R(K_{X'_k}+\Theta _k)$.
By Claim \ref{c-1}, it follows that $\kappa (K_{X'_K}+\Theta _K)=\kappa  (K_{X'_k}+\Theta _k)$ and 
there exists an integer $m_0>0$ such that  $$h^0(m(K_{X'_K}+\Theta _K))=h^0(m(K_{X'_k}+\Theta _k))\qquad \forall m\in m_0\mathbb N.$$ 
By semicontinuity, we then have $$h^0(m(K_{X_k}+B_k))\geq h^0(m(K_{X_K}+B_K))\geq h^0(m(K_{X'_K}+B'_K))\geq\qquad$$ $$\qquad h^0(m(K_{X'_K}+\Theta _K))=h^0(m(K_{X'_k}+\Theta _k))=h^0(m(K_{X_k}+B_k))$$ and hence
$h^0(m(K_{X_k}+B_k))=h^0(m(K_{X_K}+B_K))$. 
The equality
 $\kappa (K_{X_k}+B_k)=\kappa (K_{X_K}+B_K)$ follows similarly.

\end{proof}
\begin{corollary}\label{c-fingen} Let $(X,B)$ be a klt pair which is log smooth, projective of dimension $2$ over a DVR $R$ with perfect residue field $k$ of characteristic $p>0$ and perfect fraction field $K$. If $K_X+B$ is $\Q$-Cartier, $H$ is an ample divisor on $X$, and either $\kappa (K_{X_k}+B_k)\in \{0,2\}$ or $\kappa (K_{X_k}+B_k)=1$ and $B_k$ is big over ${\rm Proj} R (K_{X_k}+B_k)$, then
$R(K_X+B,K_X+B+H)$ is finitely generated.\end{corollary}
\begin{proof} By Theorem \ref{t-mmps}, $R(K_{X_k}+B_k,K_{X_k}+B_k+H_k)$ is finitely generated and hence there is an sequence of rational numbers $0=q_0<q_1<q_2<\ldots <q_n=1$ and an integer $m$ such that $R(m(K_{X_k}+B_k+q_iH_k),m(K_{X_k}+B_k+q_{i+1}H_k))$ is generated in degree 1 i.e. by $H^0(m(K_{X_k}+B_k+q_iH_k))$ and $H^0(m(K_{X_k}+B_k+q_{i+1}H_k))$.
We may assume that each $m(K_{X}+B+q_iH)$ is Cartier.
By Theorem 
  \ref{t-definv}, after replacing $m$ by a multiple, we may assume that $$H^0(m(K_X+B+q_iH))\to H^0(m(K_{X_k}+B_k+q_iH_k))$$ is surjective. 
Therefore, the induced map $$S^kH^0(m(K_X+B+q_iH))\otimes S^jH^0(m(K_X+B+q_{i+1}H))\to$$ $$S^kH^0(m(K_{X_k}+B_k+q_iH_k))\otimes S^jH^0(m(K_{X_k}+B_k+q_{i+1}H_k))\to$$ $$H^0(mk(K_{X_k}+B_k+q_iH_k)+mj(K_{X_k}+B_k+q_{i+1}H_k))$$ is surjective.
By Nakayama's lemma, $$S^kH^0(m(K_{X}+B+q_iH))\otimes S^jH^0(m(K_{X}+B+q_{i+1}H))\to $$
$$H^0(mk(K_{X}+B+q_iH)+mj(K_{X}+B+q_{i+1}H))$$ is surjective and so $R(m(K_X+B+q_iH),m(K_X+B+q_{i+1}H))$ is finitely generated.  It follows easily that $R(m(K_X+B),m(K_X+B+H))$ is finitely generated. By \cite[2.25]{CL13}  $R(K_X+B,K_X+B+H)$ is finitely generated. 
\end{proof}
\begin{remark}\label{r_fg} By the same arguments in the proof of Corollary \ref{c-fingen}, one sees that if $\kappa (K_{X_k}+B_k)=2$ and $H_1,\ldots ,H_l$ are ample divisors, then $$R(K_X+B, K_X+B+H_1,\ldots , K_X+B+H_l):=$$
$$\bigoplus _{(m_0,\ldots ,m_l)\in \mathbb N^{l+1}}H^0(m_0(K_X+B)+m_1( K_X+B+H_1)+\ldots +m_l(K_X+B+H_l))$$
is finitely generated.
\end{remark}
\begin{corollary}\label{c-mmp} Let $(X,B)$ be a klt pair which is log smooth, projective of dimension $2$ over a DVR $R$ with perfect residue field $k$ of characteristic $p>0$ and perfect fraction field $K$. If $K_X+B$ is $\Q$-Cartier and $H$ is a sufficiently ample divisor on $X$ such that $H$ is general in $N^1(X/R)$,  then (after possibly extending $R$) we may run the $K_X+B$ mmp over $R$ with scaling of $H$ which  terminates with a $K_X+B+\tau H$ minimal model $X\dasharrow \bar X$ over $R$ where $\tau ={\rm inf}\{t\geq 0|\kappa (K_{X_k}+B_k+tH_k)\geq 0\}$.\end{corollary}
\begin{proof} Suppose that $X\dasharrow X_1 \dasharrow \ldots \dasharrow X_i$ is a sequence of steps in the $K_X+B$ mmp over $R$ with scaling of $H$, in particular $X_i$ is $\Q$-factorial and there exists a rational number $\lambda_i \in \Q_{\geq 0}$ such that $K_{X_i}+B_i+\lambda _i H_i$ is nef over $R$ where $B_i$ and $H_i$ are the strict transforms of $B$ and $H$ on $X_i$. 
Let $\lambda _{i+1}={\rm inf}\{\lambda \geq 0|K_{X_i}+B_i+\lambda  H_i\ {\rm is\ nef}\}$. If $\lambda _{i+1}=0$ we are done, and therefore we may assume that $0<\lambda _{i+1}\leq \lambda _i$.

Suppose that $K_{X_k}+B_k+\lambda _{i+1} H_k$ is big. Since bigness is an open condition, $K_{X_k}+B_k+\lambda ' H_k$ is also big for some rational number $\lambda '<\lambda _{i+1} $. By Corollary \ref{c-fingen}, $R(K_X+B+\lambda ' H, K_X+B+\lambda _i H)$ is finitely generated and hence so is $R(K_{X_i}+B_i+\lambda ' H_i, K_{X_i}+B_i+\lambda _{i}H_i)$. It then follows easily that $\lambda _{i+1}\in \Q$.
 Since $R(K_{X_i}+B_i+\lambda _{i+1}H_i)$ is finitely generated, there exists a morphism 
$\mu _i:X_i\to Z_i:={\rm Proj}(R(K_{X_i}+B_i+\lambda _{i+1}H_i))$. 

We claim that $\rho (X_i/Z_i)=1$. To see this, pick ample $\Q$-divisors $H^1,\ldots , H^l$ on $X$ such that $||H-H^i||\ll 1$ in $N^1(X)$ and  $H^1,\ldots , H^l$ span $N^1(X)$.
Let $\mathcal C$ be the cone spanned by $K_X+B+\lambda ' H, K_X+B+H^1,\ldots , K_X+B+H^l$, then by Remark
\ref{r_fg}, the adjoint ring
$$R( K_X+B+\lambda ' H, K_X+B+H^1,\ldots , K_X+B+H^l),$$ is finitely generated.
Let $\mathcal C=\cup _{j=1}^s\mathcal C_j$ be the finite decomposition of $\mathcal C$ in to rational cones given by \cite[3.5]{CL13}. 
Let $K_X+B+\mathcal H$  be  an element in the interior of one of the maximal dimensional cones $\mathcal C _j$, then following \cite[3.3]{HM13} there is an induced birational map $g:X\dasharrow Y={\rm Proj} R(K_X+B+\mathcal H)$ to a log terminal model of $K_X+B+\mathcal H$. Since $H$ is general in $N^1(X)$, it follows that $\lambda _{i+1}$ is determined by the intersection of the segment $[K_X+B+\lambda 'H,K_X+B+H]$ with a codimension one face of the decomposition $\mathcal C=\cup _{j=1}^s\mathcal C_j$. We may assume that this face is shared by maximal dimensional cones $\mathcal C _i$ and $\mathcal C _{i+1}$ where $K_X+B+\lambda _{i+1}H$ corresponds to a general element of $ \mathcal C _{i}\cap \mathcal C _{i+1}$ and  if $K_X+B+\mathcal H\in \mathcal C _i^\circ$, then $X_i\cong {\rm Proj}   R(K_X+B+\mathcal H)$.
Following \cite[3.3(4)]{HM13}, there is a morphism $\mu _i:X_i\to Z_i:={\rm Proj}(R(K_{X_i}+B_i+\lambda _{i+1}H_i))$ of relative Picard number 1.

If ${\rm Ex}(\mu _i)$ is a divisor, then $Z_i$ is $\Q$-factorial, we let $X_{i+1}=Z_i$ and 
$X_i\to X_{i+1}$ is the required divisorial contraction. 
If ${\rm Ex}(\mu _i)$ is not a divisor, then $X_i\to Z_i$ is a flipping contraction and we let $X_{i+1}={\rm Proj}(R(K_{X_i}+B_i+(\lambda _{i+1}-\delta)H_i))$ for any $0<\delta \ll 1$. Note that by the finite generation of $R(K_X+B+\lambda ' H, K_X+B+H)$, it follows that $X_{i+1}$ does not depend on $\delta$. 
After replacing $X_i$ by $X_{i+1}$, we may repeat the above procedure. 

	Suppose now that $K_{X_k}+B_k+\lambda _{i+1} H_k$ is not big and $K_{X_k}+B_k+\lambda _{i} H_k$ is  big. Note that by what we have observed above $X\dasharrow X_{i}$ is a minimal model for $(X,B+\lambda H)$ for any $\lambda _{i+1}<\lambda <\lambda _i$ and so it is a $\Q$-factorial weak log canonical model for $(X,B+\lambda _{i+1} H)$. By construction, $X\dasharrow X_{i}$ is $K_X+B+\lambda _{i+1} H$ negative and hence  $X\dasharrow X_{i}$ is a minimal model for $(X,B+\lambda _{i=1}H)$.
\end{proof}

\end{document}